\documentclass[12pt,oneside]{amsart}
\usepackage{amssymb,verbatim,enumerate,ifthen}
\usepackage{cite}
\usepackage[mathscr]{eucal}
\usepackage[utf8]{inputenc}
\usepackage[T1]{fontenc}
\usepackage{esint}
\usepackage[marginparwidth=2.4cm]{geometry}

\newtheorem{thm}{Theorem}[section]
\newtheorem{cor}[thm]{Corollary}
\newtheorem{prop}[thm]{Proposition}
\newtheorem{lem}[thm]{Lemma}

\theoremstyle{definition}
\newtheorem{defin}[thm]{Definition}

\newtheorem*{xrem}{Remark}

\numberwithin{equation}{section}
\def\eq#1{{\rm(\ref{#1})}}
\def\Eq#1#2{\ifthenelse{\equal{#1}{*}}
  {\begin{equation*}\begin{aligned}[]#2\end{aligned}\end{equation*}}
  {\begin{equation}\begin{aligned}[]\label{#1}#2\end{aligned}\end{equation}}}

\def\A{\mathscr{A}}

\def\D{\mathscr{D}}
\def\E{\mathscr{E}}
\def\G{\mathscr{G}}
\def\M{\mathscr{M}}
\def\P{\mathscr{P}}
\newcommand{\operator}[1]{\mathop{\vphantom{\sum}\mathchoice
{\vcenter{\hbox{\LARGE $#1$}}}
{\vcenter{\hbox{\large $#1$}}}{#1}{#1}}\displaylimits}

\def\Mm{\operator{\mathscr{M}}}
\def\Mst_#1^#2{\operator{\mathscr{M}_{\mbox{\scriptsize$\#$}}\!\!}_{#1}^{#2}\,\,}
 
\def\Ef{\operator{\E_{\!\mbox{\scriptsize$f$}}}}
\newcommand\R{\mathbb{R}}

\newcommand\N{\mathbb{N}}
\newcommand\Z{\mathbb{Z}}
\newcommand\Q{\mathbb{Q}}
\newcommand{\QA}[1]{\A_{#1}}

\newcommand{\abs}[1]{\left| #1 \right| }
\newcommand{\ceil}[1]{\left\lceil #1 \right\rceil}

\newcommand{\vone}{\textbf{1}}
\DeclareMathOperator{\sign}{sign}

\newcommand{\Hc}[2][SKIPPED]{
\ifthenelse{\equal{#1}{SKIPPED}}
  {
    \ifthenelse{\equal{#2}{}}
      {\mathscr{H}}
      {\mathscr{H}(#2)}
  }
  {
    \ifthenelse{\equal{#2}{}}
      {\mathscr{H}_{#1}}
      {\mathscr{H}_{#1}(#2)}
  }
}
\newcommand{\Est}[2][SKIPPED]{
\ifthenelse{\equal{#1}{SKIPPED}}
  {
    \ifthenelse{\equal{#2}{}}
      {\mathscr{C}}
      {\mathscr{C}(#2)}
  }
  {
    \ifthenelse{\equal{#2}{}}
      {\mathscr{C}_{#1}}
      {\mathscr{C}_{#1}(#2)}
  }
}
\title
{On Hardy type inequalities for weighted quasideviation means}

\author{Zsolt P\'ales}
\address{Institute of Mathematics, University of Debrecen, Pf.\ 400, 4002 Debrecen, Hungary}
\email{pales@science.unideb.hu}

\author{Pawe\l{} Pasteczka}
\address{Institute of Mathematics, Pedagogical University of Krak\'ow,  Podchor\k{a}\.{z}ych str 2, 30-084 Krak\'ow, Poland}
\email{pawel.pasteczka@up.krakow.pl}

\thanks{The research of the first author was supported by the EFOP-3.6.1-16-2016-00022 and the EFOP-3.6.2-16-2017-00015 projects. These projects are co-financed by the European Union and the European Social Fund.}



\usepackage{color}

\begin{document}
\begin{abstract}
 Using recent results concerning the homogenization and the Hardy property of weighted means, we establish sharp Hardy constants for concave and monotone weighted quasideviation means and for a few particular subclasses of this broad family. More precisely, for a mean $\mathscr{D}$ like above and a sequence $(\lambda_n)$ of positive weights such that $\lambda_n/(\lambda_1+\dots+\lambda_n)$ is nondecreasing, we determine the smallest number $H \in (1,+\infty]$ such that
 $$
 \sum_{n=1}^\infty
 \lambda_n \mathscr{D}\big((x_1,\dots,x_n),(\lambda_1,\dots,\lambda_n)\big) \le H \cdot \sum_{n=1}^\infty \lambda_n x_n \text{ for all }x \in \ell_1(\lambda).
 $$
It turns out that $H$ depends only on the limit of the sequence $(\lambda_n/(\lambda_1+\dots+\lambda_n))$ and the behaviour of the mean $\mathscr{D}$ near zero.
\end{abstract}
\maketitle

\section{Introduction}
In 1920's several authors, motivated by a conjecture of Hilbert, proved that 
\Eq{H}{
  \sum_{n=1}^\infty \P_p(x_1,\dots,x_n) \le C(p) \sum_{n=1}^\infty x_n
}
for every sequences $(x_n)_{n=1}^\infty$ with positive terms, where $\P_p$ denotes the $p$-th \emph{power mean} 
(extended to the limiting cases $p=\pm\infty$),
\Eq{*}{
C(p):=
\begin{cases} 
1 & p=-\infty, \\
(1-p)^{-1/p}&p \in (-\infty,0) \cup (0,1), \\ 
e & p=0, \\
\infty & p\in[1,\infty],
\end{cases} 
}
and this constant is sharp, i.e., it cannot be diminished. 

The first result of this type with a nonoptimal constant was established by Hardy in \cite{Har20a}. Later this result was improved and extended by Landau \cite{Lan21}, Knopp \cite{Kno28}, and Carleman \cite{Car32} whose results are summarized in the inequality \eq{H}. Meanwhile, Copson \cite{Cop27} adopted Elliott's \cite{Ell26} proof of the Hardy inequality and showed (in an equivalent form) that if $\P_p(x,\lambda)$ denotes the $p$-th \emph{$\lambda$-weighted power mean} of the vector $x$, then
\Eq{E:EllCop}{
\sum_{n=1}^{\infty} \lambda_n\P_p\big((x_1,\dots,x_n),(\lambda_1,\dots,\lambda_n)\big) \le C(p) \sum_{n=1}^{\infty} \lambda_nx_n
}
for all $p\in (0,1)$, and sequences $(x_n)_{n=1}^\infty$ and $(\lambda_n)_{n=1}^\infty$ with positive terms. For more details about the history of the developments related to Hardy type inequalities, see papers Pe\v{c}ari\'c--Stolarsky \cite{PecSto01}, Duncan--McGregor \cite{DucMcG03}, and the book of Kufner--Maligranda--Persson \cite{KufMalPer07}. 

Obviously, the constant $C(p)$ is sharp if we require the inequality to be valid for all positive sequences $\lambda$ and $x$. One of the main goal of this presentation is to determine the best possible constant $C_\lambda(p)$ such that the inequality \eq{E:EllCop} be valid with $C(p)$ replaced by $C_\lambda(p)$ for all positive sequences $x$. Moreover, we will extend this result also for the case $p\leq0$. In fact, under some additional assumptions, we will show that $C_\lambda(p)$ is function of $p$ and the limit of the sequence $\big(\frac{\lambda_n}{\lambda_1+\cdots+\lambda_n}\big)$. On the other hand, our results will be developed not only for power means, but in a much larger class of weighted means, in the class of weighted quasideviation means which includes quasiarithmetic and also Gini means. The motivation for this paper originates from the paper \cite{PalPas18a} related to the nonweighted and homogeneous case.

\section{Weighted means}

For $n\in\N$, define the set of $n$-dimensional real weight vectors $W_n$ by
\Eq{*}{
  W_n:=\{(\lambda_1,\dots,\lambda_n)\in \R^n\mid\lambda_1,\dots,\lambda_n\geq0,\,\lambda_1+\dots+\lambda_n>0\}
}
and let 
\Eq{*}{
  W_0:=\{(\lambda_n)_{n\in\N}\mid\lambda_1>0 \mbox{ and } \lambda_2,\dots,\lambda_n,\ldots\geq0\}.
}
Now we recall the concept of a weighted mean as it was introduced in the paper \cite{PalPas18b}. 

For a given subinterval $I \subset \R$, \emph{a weighted mean on $I$}  is a function 
\Eq{*}{
\M \colon \bigcup_{n=1}^{\infty} I^n \times W_n \to I
}
which is nullhomogeneous in the weights, admits the reduction principle, the mean value property, and the elimination principle (see \cite{PalPas18b} for the details). For $n\in\N$ and $(x,\lambda) \in I^n\times W_n$, we will frequently use the sum type abbreviation:
\Eq{*}{
\Mm_{i=1}^n(x_i,\lambda_i):=\M\big((x_1,\dots,x_n),(\lambda_1,\dots,\lambda_n)\big).
}

Let us now introduce some important properties of weighted means. A weighted mean $\M$ is said to be \emph{symmetric}, if for all $n \in \N$, 
$(x,\lambda) \in I^n\times W_n$, and $\sigma \in S_n$, 
\Eq{*}{
\M(x,\lambda) =\M(x\circ\sigma,\lambda\circ\sigma).
}
We will call a weighted mean $\M$ \emph{Jensen concave} if, for all $n \in \N$, $x,y \in I^n$ and $\lambda \in W_n$,
\Eq{E:JF2}{
\M \Big( \frac{x+y}2 , \lambda \Big) \ge\frac12 \big( \M(x,\lambda)+\M(y,\lambda) \big).
}
If the above inequality holds with reversed inequality sign, then we speak about the \emph{Jensen convexity} of $\M$. Using that the mapping 
$x\mapsto\M(x,\lambda)$ is locally bounded, the Bernstein--Doetsch Theorem \cite{BerDoe15} implies that $\M$ is in fact concave or convex, respectively. 

A weighted mean $\M$ is said to be \emph{monotone} (or \emph{nondecreasing}) if, for all $n \in \N$ and $\lambda \in W_n$, the mapping $x_i \mapsto \M(x,\lambda)$ is nondecreasing for all $i \in \{1,\dots,n\}$.

Assuming that $I$ is a subinterval of $\R_+$, we call a weighted mean $\M$ \emph{homogeneous}, if for all $t>0$, $n\in\N$ and 
$(x,\lambda)\in \big(I\cap\frac1t I\big)^n\times W_n$,
\Eq{*}{
  \M(tx,\lambda)=t\M(x,\lambda).
}
For a given subinterval $I$ of $\R_+$ with $\inf I=0$ and a weighted mean $\M$ on $I$, we define two
functions $\M_\#,\M^\# \colon \bigcup_{n=1}^{\infty} \R_+^n \times W_n \to \R_+$ by
\Eq{*}{
  \M_\#(x,\lambda):=\liminf_{t\to 0^+}\tfrac1t\M(tx,\lambda) \qquad\mbox{and}\qquad \M^\#(x,\lambda):=\limsup_{t\to 0^+}\tfrac1t\M(tx,\lambda).
}
We call $\M_\#$ and $\M^\#$ the \emph{lower and upper homogenization} of the weighted mean $\M$, respectively. It is obvious that $\M_\#$ and $\M^\#$ are homogeneous weighted means on $\R_+$, furthermore, we have the inequality $\M_\#\leq\M^\#$ on $\bigcup_{n=1}^{\infty}\R_+^n \times W_n$. It is also easy to see that if $\M$ is symmetric (monotone), then also $\M_\#$ and $\M^\#$ are symmetric (monotone). Moreover, in the case when $\M$ is concave, we have a few additional properties.

\begin{lem}[\cite{PalPas19b}, Theorem 2.1]\label{lem:AMH2.1}
 Let $I$ be a subinterval of $\R^+$ with $\inf I = 0$ and $\M$ be a
Jensen concave weighted mean on $I$. Then $\M_\# = \M^\#$ and these means
are also Jensen concave. In addition, $\M \le \M_\# = \M^\#$ on the domain of $\M$.
\end{lem}

\medbreak

In what follows, we recall several particular classes of weighted means. 
For a parameter $p\in\R$, define the \emph{weighted power mean} $\P_p\colon \bigcup_{n=1}^{\infty} \R_+^n \times W_n \to \R_+$ by
\Eq{*}{
\P_p (x,\lambda):= 
\begin{cases} 
\left(\dfrac{\lambda_1x_1^p+\cdots+\lambda_nx_n^p}{\lambda_1+\cdots+\lambda_n} \right)^{1/p} &\quad \text{ if } p \ne 0, \\[4mm]                                                         
\left(x_1^{\lambda_1}\cdots x_n^{\lambda_n} \right)^{1/(\lambda_1+\cdots+\lambda_n)} &\quad \text{ if } p = 0.
\end{cases}
}
In a more general setting, we can define weighted quasiarithmetic means in the spirit of \cite{HarLitPol34}. Given an interval $I$ and a continuous strictly monotone function $f \colon I \to \R$, the \emph{weighted quasiarithmetic mean} 
$\QA{f} \colon \bigcup_{n=1}^{\infty} I^n \times W_n \to I$ is defined by
\Eq{QA}{
\QA{f}(x,\lambda):= f^{-1} \left( \frac{\lambda_1 f(x_1)+\cdots+\lambda_nf(x_n)}{\lambda_1+\cdots+\lambda_n} \right).
}
Another important generalization of power means was introduced in the paper \cite{Gin38}. For two real parameters $p,q$, the Gini mean $\G_{p,q}\colon \bigcup_{n=1}^{\infty} \R_+^n \times W_n \to \R_+$ is defined by
\Eq{*}{
\G_{p,q} (x,\lambda):= 
\begin{cases} 
\left(\dfrac{\lambda_1x_1^p+\cdots+\lambda_nx_n^p}{\lambda_1x_1^q+\cdots+\lambda_nx_n^q} \right)^{\frac{1}{p-q}} &\quad \text{ if } p \ne q, \\[4mm]                                                         
\exp\left(\dfrac{\lambda_1x_1^p\log x_1+\cdots+\lambda_nx_n^p\log x_n}{\lambda_1x_1^p+\cdots+\lambda_nx_n^p} \right) &\quad \text{ if } p = q.
\end{cases}
}

In a sequence papers, further generalizations were obtained: \emph{Bajraktarevi\'c means} \cite{Baj58}, \emph{deviation (or Dar\'oczy) means} \cite{Dar71b} and \emph{quasideviation means} \cite{Pal82a}. For more details, we just refer the reader to a series of papers by Losonczi \cite{Los70a,Los71a,Los71b,Los71c,Los73a,Los77} (for Bajraktarevi\'c means), Dar\'oczy \cite{Dar71b,Dar72b}, Dar\'oczy--Losonczi \cite{DarLos70}, Dar\'oczy--P\'ales \cite{DarPal82,DarPal83} (for deviation means), P\'ales \cite{Pal82a,Pal83b,Pal84a,Pal85a,Pal88a,Pal88d,Pal88e} (for deviation and quasideviation means) and P\'ales--Pasteczka \cite{PalPas19b} (for semideviation means).

In what follows, we recall the notions of a quasideviation and the related  weighted quasideviation mean (cf. \cite{Pal82a}, \cite{Pal89b} and \cite{PalPas19b}).

\begin{defin} A function $E\colon I\times I\to\R$ is said to be a \emph{quasideviation} if
\begin{enumerate}[(a)]
 \item for all elements $x,y\in I$, the sign of $E(x,y)$ coincides with that of $x-y$,
 \item for all $x\in I$, the map $y\mapsto E(x,y)$ is continuous and,
 \item for all $x<y$ in $I$, the mapping $(x,y)\ni t\mapsto \frac{E(y,t)}{E(x,t)}$ is strictly increasing.
\end{enumerate}
By the results of the paper \cite{Pal82a},
for all $n\in\N$ and $(x,\lambda)\in I^n\times W_n$, the equation
\Eq{e}{
  \lambda_1 E(x_1,y)+\cdots+\lambda_n E(x_n,y)=0
}
has a unique solution $y$, which will be called the \emph{$E$-quasideviation mean} of $(x,\lambda)$ and denoted by $\D_E(x,\lambda)$.

One can easily notice that power means, quasiarithmetic means, Gini means are quasideviation means. 

We say that a quasideviation $E\colon I\times I\to\R$ is \emph{normalizable} if, for all $x\in I$, the function $y\mapsto E(x,y)$ is differentiable at $x$ and the mapping $x\mapsto\partial_2E(x,x)$ is strictly negative and continuous on $I$. The normalization $E^*\colon I\times I\to\R$ of $E$ is defined by
\Eq{*}{
  E^*(x,y):=\frac{E(x,y)}{-\partial_2E(y,y)} \qquad(x,y\in I).
}
The quasideviation means generated by $E$ and $E^*$ are identical. In \cite[Lemma 5.1]{PalPas19b} we proved that, for a normalized quasideviation $E$, the partial derivative $\partial_2E$ is identically equal to $-1$ on the diagonal of $I\times I$, hence $E^*$ is also a normalizable quasideviation and $(E^*)^*=E^*$ holds.
\end{defin}

The following two results of the papers \cite{PalPas18a} and \cite{PalPas19b}
are instrumental for us.

\begin{lem}[\!\!\cite{PalPas18a}, Theorem~2.3]\label{lem:MIA2.3}
Let $f\colon\R_+\to\R$ be concave such that $\sign(f(x))=\sign(x-1)$ for all $x\in\R_+$. Then the function $E\colon\R_+^2\to\R$ defined by $E(x,y):=f\big(\frac xy\big)$ is a quasideviation and the weighted quasideviation mean $\E_f:=\D_E$ is homogeneous, continuous, nondecreasing and concave.
\end{lem}

\begin{lem}[\!\!\cite{PalPas19b}, Theorem 6.3]\label{lem:AMH6.3}
Let $E : I \times I \to \R$ be a normalizable quasideviation such that $E^*$ is concave. Assume that $\lim_{t \to 0^+}E^*(xt, t) = 0$ for all $x \in \R_+$. Then, for all $x \in \R_+$, the limit 
\Eq{hE}{
h_E(x) := \lim_{t \to 0^+} t^{-1} E^*(xt,t)
}
exists, $\sign(h_E(x))=\sign(x-1)$, and the function $h_E : \R_+ \to \R$ so defined is concave and nondecreasing on $\R_+$, and is strictly increasing on $(0, 1)$. 
Furthermore, the weighted quasideviation mean $\D_E$ is Jensen concave, monotone, and
\Eq{*}{
\E_{h_E}=(\D_E)_\#=(\D_E)^\#.
}
\end{lem}

\section{Hardy type inequalities for general weighted means}

We recall several definitions and results of the papers \cite{PalPas19a} and \cite{PalPas19b}. Throughout the rest of the paper, let $I$ be an interval with $\inf I=0$.

\begin{defin}[Weighted Hardy property]
For a weighted mean $\M$ on $I$ and a weight sequence $\lambda \in W_0$, let $C$ be the smallest extended real number such that 
\Eq{*}{
\sum_{n=1}^{\infty} \lambda_n \cdot \Mm_{i=1}^n\big(x_i,\lambda_i\big) \le C \cdot \sum_{n=1}^{\infty} \lambda_nx_n \qquad\text{ for all sequences } (x_n) \text{ in $I$}.
}
We call $C$ to be the \emph{$\lambda$-weighted Hardy constant of $\M$} or the \emph{$\lambda$-Hardy constant of $\M$} and denote it by $\Hc[\lambda]\M$. Whenever this constant is finite, then $\M$ is called a \emph{$\lambda$-weighted Hardy mean} or simply a \emph{$\lambda$-Hardy mean}.
\end{defin}

Extending some previous results by Elliott \cite{Ell26} and Copson \cite{Cop27}, we have obtained in \cite{PalPas19a} that, in a large class of weighted means, the Hardy constant corresponding to the weight sequence $\vone:=(1,1,\dots)$ is the maximal one.
\begin{thm}
\label{thm:mu1}
For every symmetric and monotone weighted mean $\M$ on $I$, we have
\Eq{*}{
\Hc[\vone]\M=\sup_{\lambda \in W_0} \Hc[\lambda]\M.
}
\end{thm}

The following lemma from \cite{PalPas19a} will be used.
\begin{lem} 
\label{lem:5}
Let $\M$ be a weighted mean on $I$ and $\lambda\in W_0$. Then, for all $n\in \N$ and $x \in I^n$,
\Eq{FinWH}{
\sum_{i=1}^n \lambda_i \cdot \Mm_{j=1}^i \big(x_j,\lambda_j\big)\le \Hc[\lambda]\M \sum_{i=1}^n \lambda_ix_i.
}
\end{lem}

Based on this lemma and Lemma~\ref{lem:AMH2.1}, we can compare the $\lambda$-Hardy constant of the weighted mean $\M$ and its lower homogenization $\M_\#$.

\begin{thm} 
\label{thm:M*}
Let $\M$ be a weighted mean on $I$. Then, for all $\lambda\in W_0$,
\Eq{M-M*}{
  \Hc[\lambda]{\M_\#}\leq \Hc[\lambda]\M.
}
If, in addition, $\M$ is Jensen concave, then \eq{M-M*} holds with equality.
\end{thm}

\begin{proof} 
Let $(x_m)$ be a sequence in $\R_+$. For any fixed $n\in\N$, there exists a positive number $\tau_n$ such that $t(x_1,\dots,x_n)\in I^n$ for $t\in(0,\tau_n]$. Using Lemma~\ref{lem:5}, it follows that
\Eq{*}{
\sum_{i=1}^n \lambda_i \cdot \Mm_{j=1}^i \big(tx_j,\lambda_j\big)\le \Hc[\lambda]\M \sum_{i=1}^n \lambda_itx_i.
}
Dividing by $t\in(0,\tau_n]$, and then taking the liminf of the left hand side of the inequality so obtained as $t\to0^+$, (by the superadditivity of the liminf operation), we arrive at
\Eq{*}{
\sum_{i=1}^n \lambda_i \cdot \liminf_{t\to0^+}\frac1t\Mm_{j=1}^i \big(tx_j,\lambda_j\big)\le \Hc[\lambda]\M \sum_{i=1}^n \lambda_ix_i.
}
This inequality is equivalent to 
\Eq{*}{
\sum_{i=1}^n \lambda_i \cdot \Mst_{j=1}^i \big(x_j,\lambda_j\big)\le \Hc[\lambda]\M \sum_{i=1}^n \lambda_ix_i.
}
Finally, passing the limit $n\to\infty$ in the above inequality, we get
\Eq{*}{
\sum_{n=1}^{\infty} \lambda_n \cdot \Mst_{i=1}^n\big(x_i,\lambda_i\big) 
\le \Hc[\lambda]\M \cdot \sum_{n=1}^{\infty} \lambda_nx_n,
}
which proves that $\Hc[\lambda]{\M_\#}\leq \Hc[\lambda]\M$.

If, additionally, $\M$ is Jensen concave, then, by Lemma~\ref{lem:AMH2.1}, the comparison inequality $\M\leq\M_\#$ is valid, and hence, \eq{M-M*} must hold with equality, indeed.
\end{proof}

The following result of the paper  \cite{PalPas19a}, which is a weighted analogue of \cite[Thm 3.3]{PalPas16}, provides a lower bound for the Hardy constant $\Hc[\lambda]\M$.

\begin{lem} 
Let $\M$ be a weighted mean on $I$, $\lambda\in W_0$, and $(x_n)_{n=1}^\infty$ be a sequence of elements in $I$.
If\, $\sum_{n=1}^\infty \lambda_nx_n=\infty$, then 
\Eq{*}{
\Hc[\lambda]\M \ge \liminf_{n \to \infty} \frac{1}{x_n} \Mm_{i=1}^n \big(x_i,\lambda_i\big).
}
\end{lem}

By taking $x_n:=\frac{y}{\Lambda_n}$ for a fixed $y\in \lambda_1I$ in the above theorem, the first inequality of the following consequence was deduced in \cite{PalPas19a}. The second inequality is an application of the first one to the mean $\M_\#$ and Theorem~\ref{thm:M*}.

\begin{cor} 
\label{cor:HardyconstantforKedlaya}
Let $\M$ be a weighted mean on $I$ and $\lambda \in W_0$ be a weight sequence with $\sum_{n=1}^\infty\lambda_n=\infty$. 
Then we have the following two lower estimates for the $\lambda$-Hardy constant $\Hc[\lambda]\M$:
\Eq{*}{
\Hc[\lambda]\M 
  \ge\sup_{y\in\lambda_1I} \liminf_{n \to \infty} \frac {\Lambda_n}y \cdot \Mm_{k=1}^n \Big(\frac{y}{\Lambda_k},\lambda_k\Big)
  =:\Est[\lambda]\M
}
and 
\Eq{*}{
\Hc[\lambda]\M 
  \ge \liminf_{n \to \infty} \Mst_{k=1}^n \Big(\frac{\Lambda_n}{\Lambda_k},\lambda_k\Big)
  =\Est[\lambda]{\M_\#}.
}
\end{cor}

Finally let us recall one of key results from \cite{PalPas19a}. 

\begin{prop}[\cite{PalPas19a}, Corollary 4.3]
\label{prop:BJ}
Let $\M$ be a symmetric, monotone and Jensen-concave weighted mean which is continuous in the weights and $\lambda \in W_0$ such that $\big(\tfrac{\lambda_n}{\lambda_1+\dots+\lambda_n}\big)_{n=1}^\infty$ is nonincreasing. Then $\Hc[\lambda]\M \le \Est[\lambda]\M$. Furthermore, if $\sum_{n=1}^\infty\lambda_n=\infty$, then $\Hc[\lambda]\M = \Est[\lambda]\M$.
\end{prop}

\section{Auxiliary Results}

In this section we prove a number of results which will be instrumental in the forthcoming sections.
Throughout this section, let $\lambda \in W_0$ be a fixed weight sequence and $\Lambda_n:=\lambda_1+\cdots+\lambda_n$ for $n\in\N$.

\begin{lem} 
\label{lem:2}
The sequence $(\Lambda_n)$ and the series $\sum\lambda_n/\Lambda_n$ are equi-convergent (either both of them are convergent or both of them are divergent).
\end{lem}

\begin{proof}
If $\Lambda_\infty:=\sum_{n=1}^{\infty} \lambda_n<\infty$, then
\Eq{*}{
\sum_{n=1}^{\infty} \frac{\lambda_n}{\Lambda_n} \le \sum_{n=1}^{\infty} \frac{\lambda_n}{\Lambda_1}=\frac{\Lambda_\infty}{\Lambda_1}<\infty.
}
Conversely, if $\sum_{n=1}^{\infty} \frac{\lambda_n}{\Lambda_n}< \infty$, then there exists $n_0 \in \N$ such that for $n\ge n_0$,
$\lambda_n/\Lambda_n<\tfrac12$. Equivalently, $\Lambda_{n-1}/\Lambda_n=1-\lambda_n/\Lambda_n>\tfrac12$ for $n \ge n_0$. Thus,
\Eq{*}{
\infty > \sum_{n=n_0}^{\infty}\frac{\lambda_n}{\Lambda_n} 
\ge \frac{1}{2} \sum_{n=n_0}^{\infty}\frac{\lambda_n}{\Lambda_{n-1}}
\ge \frac{1}{2} \sum_{n=n_0}^{\infty}\int_{\Lambda_{n-1}}^{\Lambda_n}\tfrac1x\:dx
= \frac{1}{2} \int_{\Lambda_{n_0-1}}^{\Lambda_\infty}\tfrac1x\:dx.
}
As this integral is finite, we obtain $\Lambda_\infty<\infty$.
\end{proof}

\begin{lem}
\label{lem:maxk1n}
If $\lambda_n/\Lambda_n \to 0$ and $\Lambda_n\to\infty$, then 
\Eq{maxk1n}{
\lim_{n \to \infty} \frac{\max(\lambda_1,\dots,\lambda_n)}{\Lambda_n} =0.
}
\end{lem}

\begin{proof}
Fix $\varepsilon>0$. There exists $k_0 \in \N$ such that $\lambda_k/\Lambda_k \le \varepsilon$ for all $k > k_0$.
Take $n_0 \ge k_0$ such that $\Lambda_{n_0} \ge \tfrac1\varepsilon \cdot \max(\lambda_1,\ldots,\lambda_{k_0})$.
Fix $n>n_0$ arbitrarily. Then
\Eq{*}{
\frac{\lambda_k}{\Lambda_n} 
&\le  \frac{\max(\lambda_1,\ldots,\lambda_{k_0})}{\Lambda_{n_0}} \le \varepsilon, \qquad &&k \in\{1,\dots, k_0\}, \\
\frac{\lambda_k}{\Lambda_n} 
&= \frac{\lambda_k}{\Lambda_k} \cdot \frac{\Lambda_k}{\Lambda_n} \le \varepsilon \cdot 1  =\varepsilon, \qquad &&k \in \{k_0+1,\dots,n\}.
}
Therefore, 
\Eq{*}{
\frac{\max(\lambda_1,\dots,\lambda_n)}{\Lambda_n} \le \varepsilon \quad\text{ for every }\quad n \ge n_0,
}
which completes the proof of the statement.
\end{proof}

\begin{lem}\label{lem:eq-conv}
Let $\varphi \colon (0,1] \to \R$ be a continuous and nonincreasing function and $q\in(0,1)$. Then the integral $\int_0^1 \varphi$ and the series $\sum_{k=1}^{\infty} q^k \varphi\big(q^k\big)$ are equiconvergent. Furthermore,
\Eq{q}{
\frac{q}{1-q}\int_{0}^{1} \varphi
\leq \sum_{k=1}^\infty q^k \varphi \big(q^k\big) 
\leq \frac{1}{1-q}\int_{0}^{q} \varphi.
}
\end{lem}

\begin{proof}
If $\varphi$ is constant then both the integral and the series are convergent. Therefore, replacing $\varphi$ by $\varphi-\varphi(1)$ if necessary, we may assume that $\varphi(1)=0$.
Using the nonincerasingness of $\varphi$, for all $k\in\N$, we obtain 
\Eq{*}{
 q \int_{q^{k}}^{q^{k-1}} \varphi
\leq q\int_{q^{k}}^{q^{k-1}} \varphi(q^k)
= (1-q)q^k \varphi \big(q^k\big)
= \int_{q^{k+1}}^{q^{k}}\varphi(q^k)
\leq \int_{q^{k+1}}^{q^{k}} \varphi.
}
Summing up these inequalities side by side, the inequality \eq{q} follows.
which proves the integrability of $\varphi$ over $(0,1]$. This inequality also shows the equiconvergence of the integral and the series. 
\end{proof}

\begin{prop}
\label{prop:genA}
Let $\varphi \colon (0,1] \to \R$ be a continuous and monotone function. If  $\Lambda_n\to\infty$ and the sequence $\big(\frac{\lambda_n}{\Lambda_n}\big)$ is convergent with a limit $\eta$ belonging to $[0,1)$, then
\Eq{lim}{
\lim_{n \to \infty} \sum_{k=1}^n \frac{\lambda_k}{\Lambda_n} \cdot \varphi \bigg(\frac{\Lambda_k}{\Lambda_n} \bigg)=
\begin{cases}
\int\limits_0^1 \varphi(x) dx & \mbox{if } \eta=0, \\[3mm] 
\sum\limits_{k=0}^{\infty} \eta (1-\eta)^k \varphi\big((1-\eta)^k\big) & \mbox{if }\eta \in (0,1).
\end{cases}
}
\end{prop}

\begin{proof}
We may suppose without loss of generality that $\varphi$ is nonincreasing.
The equality \eq{lim} is obvious if $\varphi$ is a constant function.
Therefore, replacing $\varphi$ by $\varphi-\varphi(1)$, we also can assume that $\varphi(1)=0$ and then $\varphi$ is nonnegative. Thus the integral and the sum of the series on the right hand side of formula \eq{lim} are well-defined, however, their value could be equal to $+\infty$.

Assume first that $\eta=0$. For $n\in\N$, consider the partition $0<\frac{\Lambda_1}{\Lambda_n}<\cdots<\frac{\Lambda_n}{\Lambda_n}=1$ of the interval $[0,1]$. By Lemma~\ref{lem:maxk1n}, the mesh size of this partition tends to zero as $n\to\infty$. The sum on the left hand side of \eq{lim} is the Lebesgue integral of the step function $\varphi_n$ defined as $\varphi_n(t)=\varphi(\Lambda_k/\Lambda_n)$ for $t\in(\Lambda_{k-1}/\Lambda_n,\Lambda_k/\Lambda_n]$. Due to the inequality $\varphi_n\leq\varphi$, we have that the left hand side of \eq{lim} is smaller than or equal to the right side. To prove the reversed inequality, let $c<\int_0^1\varphi(x)dx$. Then, there exists $0<\alpha<1$ such that $c<\int_\alpha^1\varphi(x)dx$. By the continuity of $\varphi$ and \eq{maxk1n}, the sequence of functions $\varphi_n$ pointwise converges to $\varphi$ and the convergence is uniform on the interval $[\alpha,1]$. Therefore the sequence of integrals $\int_\alpha^1\varphi_n(x)dx$ converges to $\int_\alpha^1\varphi(x)dx$. Thus, for large $n$, we have that 
\Eq{*}{
  c<\int_\alpha^1\varphi_n(x)dx\leq \int_0^1\varphi_n(x)dx=
 \sum_{k=1}^n \frac{\lambda_k}{\Lambda_n} \cdot \varphi \left(\frac{\Lambda_k}{\Lambda_n} \right).
}
This proves the reversed inequality in \eq{lim} in the case $\eta=0$.

From now on let us assume that $\eta>0$. We know that
\Eq{*}{
\lim_{n \to \infty}\frac{\Lambda_{n-1}}{\Lambda_n}=1-\eta
}
and therefore, by simple induction,
\Eq{*}{
\lim_{n \to \infty}\frac{\Lambda_{n-k}}{\Lambda_n}=(1-\eta)^k,\quad k\in\N.
}
In particular,
\Eq{*}{
\lim_{n \to \infty}\frac{\lambda_{n-k}}{\Lambda_n}
&=\lim_{n \to \infty}\frac{\Lambda_{n-k}}{\Lambda_n}-\frac{\Lambda_{n-k-1}}{\Lambda_n}
= (1-\eta)^k-(1-\eta)^{k+1}=\eta \cdot (1-\eta)^k.
}
Therefore, for all $k \in \N$,
\Eq{genA4}{
\lim_{n \to \infty} \frac{\lambda_{n-k}}{\Lambda_n} \cdot \varphi \left( \frac{\Lambda_{n-k}}{\Lambda_n} \right)
=\eta\cdot (1-\eta)^k \cdot \varphi \left((1-\eta)^k\right).
}

For all $n>m\geq1$, we have
\Eq{*}{
\sum_{k=1}^n \frac{\lambda_k}{\Lambda_n} \varphi\bigg(\frac{\Lambda_k}{\Lambda_n}\bigg)=
\sum_{k=0}^{n-1} \frac{\lambda_{n-k}}{\Lambda_n} \varphi\bigg(\frac{\Lambda_{n-k}}{\Lambda_n}\bigg)
\ge \sum_{k=0}^{m} \frac{\lambda_{n-k}}{\Lambda_n} \varphi\bigg(\frac{\Lambda_{n-k}}{\Lambda_n}\bigg).
}
Thus, using \eq{genA4}, the above inequality implies
\Eq{*}{
\liminf_{n \to \infty} \sum_{k=0}^n \frac{\lambda_k}{\Lambda_n} \varphi\bigg(\frac{\Lambda_k}{\Lambda_n}\bigg)
\ge \lim_{n \to \infty} \sum_{k=0}^{m} \frac{\lambda_{n-k}}{\Lambda_n} \varphi\bigg(\frac{\Lambda_{n-k}}{\Lambda_n}\bigg) 
=\sum_{k=0}^{m} \eta(1-\eta)^k \cdot \varphi\big((1-\eta)^k\big).
}
Upon taking the limit $m\to\infty$, it follows that 
\Eq{*}{
\liminf_{n \to \infty} \sum_{k=0}^n \frac{\lambda_k}{\Lambda_n} \varphi\bigg(\frac{\Lambda_k}{\Lambda_n}\bigg)
\geq
\sum_{k=0}^{\infty} \eta(1-\eta)^k \cdot \varphi\big((1-\eta)^k\big).
}
This implies also the equality in \eq{lim} if the right-hand-side series is divergent. Therefore, in the rest of the proof, we can assume that this series is convergent.

Fix $\varepsilon>0$ and choose $k_0$ such that
\Eq{genA3}{
\sum_{k=k_0}^{\infty} \eta\cdot (1-\eta)^k \cdot \varphi \left((1-\eta)^k\right) \le \frac \varepsilon 4 \quad \text{ and } \quad \int_0^{(1-\eta)^{k_0}} \varphi(x)dx \le \frac \varepsilon 4 .
}
Moreover, by \eq{genA4}, there exists $n_0$ such that for all $k \in \{0,1,\ldots,k_0-1\}$ and $n \ge n_0$
\Eq{genA1}{
\abs{ \frac{\lambda_{n-k}}{\Lambda_n} \cdot \varphi \left( \frac{\Lambda_{n-k}}{\Lambda_n} \right) -\eta\cdot (1-\eta)^k \cdot \varphi((1-\eta)^k)}\le \frac \varepsilon {4k_0} .
}

Now, applying the nonincreasingness of $\varphi$ again, for all $n \ge k_0$,
\Eq{genA5}{
\sum_{k=k_0}^{n-1} \frac{\lambda_{n-k}}{\Lambda_n} \cdot \varphi\left(\frac{\Lambda_{n-k}}{\Lambda_n} \right) 
= \sum_{k=1}^{n-k_0} \frac{\lambda_{k}}{\Lambda_n} \cdot \varphi\left(\frac{\Lambda_{k}}{\Lambda_n} \right)
\le \int_0^{\frac{\Lambda_{n-k_0}}{\Lambda_n}} \varphi(x) dx 
}
But
\Eq{*}{
\lim_{n \to \infty} \int_0^{\frac{\Lambda_{n-k_0}}{\Lambda_n}} \varphi(x) dx =
\int_0^{(1-\eta)^{k_0}} \varphi(x) dx \le \frac \varepsilon4,
}
so there exists $n_1 \ge \max(n_0,k_0)$ such that
\Eq{*}{
\int_0^{\frac{\Lambda_{n-k_0}}{\Lambda_n}} \varphi(x) dx \le \frac \varepsilon2 \qquad\text{ for all }\qquad n \ge n_1.
}
Thus, by \eq{genA5},
\Eq{genA6}{
\sum_{k=k_0}^{n-1} \frac{\lambda_{n-k}}{\Lambda_n} \cdot \varphi\left(\frac{\Lambda_{n-k}}{\Lambda_n} \right) \le \frac \varepsilon2 \qquad\text{ for all }\qquad n \ge n_1.
}
Finally, applying \eq{genA1}, \eq{genA6}, and \eq{genA3}, for all $n \ge n_1$,
\Eq{*}{
\bigg|\sum_{k=1}^n \frac{\lambda_k}{\Lambda_n} \cdot \varphi\left(\frac{\Lambda_k}{\Lambda_n} \right) &- \sum_{k=0}^{\infty} \eta (1-\eta)^k\varphi((1-\eta)^k)\bigg|\\
 &= \abs{\sum_{k=0}^{n-1} \frac{\lambda_{n-k}}{\Lambda_n} \cdot \varphi\left(\frac{\Lambda_{n-k}}{\Lambda_n} \right) - \sum_{k=0}^{\infty} \eta (1-\eta)^k\varphi((1-\eta)^k)}\\
  &\le \sum_{k=0}^{k_0-1} \abs{\frac{\lambda_{n-k}}{\Lambda_n} \cdot \varphi\left(\frac{\Lambda_{n-k}}{\Lambda_n} \right) - \eta (1-\eta)^k \varphi((1-\eta)^k)} \\
&\quad\qquad + \sum_{k=k_0}^{n-1} \frac{\lambda_{n-k}}{\Lambda_n} \cdot \varphi\left(\frac{\Lambda_{n-k}}{\Lambda_n} \right) + \sum_{k=k_0}^{\infty} \eta (1-\eta)^k \varphi((1-\eta)^k)\\
&\le k_0 \cdot \frac \varepsilon {4k_0} + \frac\varepsilon2 + \frac \varepsilon 4 = \varepsilon.
}
This completes the proof in the case $\eta\in(0,1)$. 
\end{proof}

It is worth mentioning that \eq{q} with $q=1-\eta$ follows that
\Eq{*}{
\lim_{\eta\to0^+}\sum\limits_{k=0}^{\infty} \eta (1-\eta)^k \varphi\big((1-\eta)^k\big)
=\int\limits_0^1 \varphi(x) dx,
}
which means, that the right hand side of \eq{lim} is a continuous function of $\eta$. Applying the above result to the power function $\varphi(x)=x^{-p}$ (where $p<1$), we immediately get 

\begin{cor} \label{cor:A}
If $p<1$, furthermore, $\Lambda_n\to\infty$ and $\frac{\lambda_n}{\Lambda_n} \to \eta\in[0,1)$, then
\Eq{*}{
\lim_{n \to \infty} \sum_{k=1}^n \frac{\lambda_k}{\Lambda_n} \cdot \left(\frac{\Lambda_k}{\Lambda_n} \right)^{-p}=
\begin{cases}
\dfrac{1}{1-p} & \eta=0, \\[2mm]
\dfrac{\eta}{1-(1-\eta)^{1-p}} & \eta \in (0,1).
\end{cases}
 }
\end{cor}

\begin{lem}\label{lem:F}
Let $f:\R_+\to\R$ be a concave function such that $\sign(f(x))=\sign(x-1)$ holds for all $x\in\R_+$ and the function $x\mapsto f(1/x)$ is integrable over $(0,1]$. Then the function $F$ given by
\Eq{*}{
  F(x,q):=\sum\limits_{k=0}^{\infty} q^k f\big(q^{-k}x\big) \qquad\big((x,q)\in\R_+\times(0,1))
}
is well-defined, continuous and nondecreasing in its first variable. Furthermore, for all fixed $q\in(0,1)$, the equation $F(x,q)=0$ has a unique solution $x(q)\in(0,1)$. The mapping $x(\cdot)$ so defined is continuous, and we have the following estimates:
\Eq{It}{
  \frac{q}{1-q} \int_{0}^{1/q}f\Big(\frac{x}{t}\Big) dt
  \leq F(x,q)\leq \frac{1}{1-q} \int_{0}^{1}f\Big(\frac{x}{t}\Big) dt \qquad (x,q)\in \R_+ \times(0,1).
}
\end{lem}

\begin{proof}
By elementary considerations, it follows from the concavity and the sign properties that $f$ is nondecreasing on $\R_+$ and is strictly increasing on $(0,1)$, furthermore, it is also continuous. It also follows from the concavity that the map
\Eq{*}{
  u\mapsto\frac{f(u)-f(1)}{u-1}=\frac{f(u)}{u-1}
}
is nonincreasing on $\R_+\setminus\{1\}$. Therefore, if $1<u_0\leq u$, then
\Eq{Iu}{
  \frac{f(u)}{u}\leq\frac{f(u)}{u-1}\leq\frac{f(u_0)}{u_0-1}.
}

To show that $F$ is continuous, let $(x_0,q_0)\in\R_+\times(0,1)$ be fixed. The product $q_0^{-k}x$ is bigger than $1$ for $k\geq k_0:=1+\ceil{\log(x_0)/\log(q_0)}$. Therefore, there exist $0<x_*<x_0<x^*$ and $0<q_*<q_0<q^*<1$ such that $q^{-k}x>1$ for all $(x,q) \in V:=[x_*,x^*]\times[q_*,q^*]$ and $k\geq k_0$. The expression $\sum_{k=0}^{k_0-1} q^k f(q^{-k} x)$ being a finite sum of continuous functions is obviously continuous at $(x_0,q_0)$. Therefore, it suffices to show that tail sum
\Eq{*}{
  F_{k_0}(x,q):=\sum_{k=k_0}^{\infty} q^k f(q^{-k} x)
}
is also continuous at $(x_0,q_0)$. By the choice of $k_0$, each term is positive  for $(x,q)\in V$. By the nondecreasingness of $f$, for all $k\geq0$ and $(x,q) \in \R_+\times(0,1)$, we clearly have
\Eq{Iv}{
q^k f(q^{-k}x) 
&= \frac{1}{1-q} \int_{q^{k+1}}^{q^{k}} f\Big(\frac{x}{q^k}\Big) dt
\le \frac{1}{1-q} \int_{q^{k+1}}^{q^{k}} f\Big(\frac{x}{t}\Big) dt,\\
q^k f(q^{-k}x) 
&= \frac{q}{1-q} \int_{q^{k}}^{q^{k-1}} f\Big(\frac{x}{q^k}\Big) dt
\ge \frac{q}{1-q} \int_{q^{k}}^{q^{k-1}} f\Big(\frac{x}{t}\Big) dt.
}
Summarizing the first inequality for $k\geq k_0$, it follows that
\Eq{*}{
\sum_{k=k_0}^\infty q^k f(q^{-k}x) \le \frac{1}{1-q} \int_{0}^{q^{k_0}} f\Big(\frac{x}{t}\Big) dt < +\infty \qquad (x,q) \in V,
}
which implies that the series on the left hand side is convergent. To prove that the sum of this series (i.e., $F_{k_0}(x,q)$) is a continuous function of $(x,q)$ at $(x_0,q_0)$, it suffices to show that the convergence is uniform over $V$. 

Observe that, for $k\geq k_0$ and $(x,q)\in V$, we have
\Eq{*}{
  u_0:=\frac{x_*}{(q^*)^{k}}\leq \frac{x}{q^{k}}:=u. 
}
Now using the inequality \eq{Iu} for the above $u_0$ and $u$, it follows that
\Eq{*}{
  q^k f(q^{-k}x)&=x\frac{f(u)}{u}
  \leq x^*\frac{f(u_0)}{u_0-1}\\
  &=\frac{x^*}{x_*-(q^*)^{k}}\cdot(q^*)^kf\big((q^*)^{-k}x_*\big)
  \leq\frac{x^*}{x_*-(q^*)^{k_0}}\cdot(q^*)^kf\big((q^*)^{-k}x_*\big).
}
This inequality shows that, for $(x,q)\in V$ and $k\geq k_0$, the $k$th term of the series corresponding to $F_{k_0}(x,q)$ is majorized by a constant multiple of the corresponding term of the series for $F_{k_0}(x_*,q^*)$. Thus, in view of the Weierstrass $M$-test, the convergence of the series corresponding to $F_{k_0}(x,q)$ is uniform. By the continuity of each term of this series, it follows that the sum function is also continuous at $(x_0,q_0)$.

Thus, we have proved that $F$ is a well-defined continuous function on $\R_+ \times (0,1)$. Moreover, as $f$ is nondecreasing on $(1,\infty)$ and strictly increasing on $(0,1)$, we obtain that $F(\cdot,q)$ is nondecreasing on $(1,\infty)$ and strictly increasing on $(0,1)$ (as all terms of the sum are nondecreasing and the very first of them is strictly increasing on $(0,1)$).

Finally, in order to prove that $F(\cdot,q)$ has a unique zero note that $f(q^n) < f(q)$, and $f(q^k) < 0$ for all $k \in \{1,\dots,n-1\}$. Thus we get
\Eq{*}{
F(q^n,q) = f(q^n)+\cdots+q^{n-1}f(q)+q^n F(1,q)< f(q)+q^n F(1,q).
}
Therefore, for large $n$, we have that $F(q^n,q)<0$. This, with the easy-to-see inequality $F(1,q)>0$, implies that for all $q \in (0,1)$, the equality $F(x,q)=0$ has a solution $x=x(q) \in (0,1)$. 

To show that $x(q)$ depends continuously on $q$, let $q_0\in(0,1)$ be fixed and $0<\varepsilon<\min(x(q_0),1-x(q_0))$. Since $F(x(q_0),q_0)=0$, we have that 
\Eq{*}{
 F(x(q_0)-\varepsilon,q_0)<0<F(x(q_0)+\varepsilon,q_0).
}
By the continuity of $F$, there exists $0<\delta<\min(q_0,1-q_0)$ such that, for all $q\in(q_0-\delta,q_0+\delta)$, 
\Eq{*}{
 F(x(q_0)-\varepsilon,q)<0<F(x(q_0)+\varepsilon,q).
}
Therefore, the uniquely defined value $x(q)$ must be between $x(q_0)-\varepsilon$ and $x(q_0)+\varepsilon$, that is, $|x(q)-x(q_0)|<\varepsilon$ for all $q\in(q_0-\delta,q_0+\delta)$.

Finally, if we sum up (both) inequalities side by side in \eq{Iv} for all $k \in \{0,1,\dots\}$, we easily obtain \eq{It}.
\end{proof}

\begin{prop}\label{prop:CEf}
 Let $\lambda \in W_0$ with $\Lambda_n\to\infty$, $\lambda_n/\Lambda_n \to \eta\in[0,1)$ and let $f:\R_+\to\R$ be a concave function such that $\sign(f(x))=\sign(x-1)$ holds for all $x\in\R_+$ and the function $x\mapsto f(1/x)$ is integrable over $(0,1]$. Then $c:=\Est[\lambda]{\E_f}$ is the unique solution of the equation 
\Eq{E:CEf}{
\int_0^1 f\Big(\frac1{cx}\Big)\:dx=0 &\qquad \text{ for }\eta=0,\\
\sum\limits_{k=0}^{\infty} (1-\eta)^k f\Big(\frac{1}{c(1-\eta)^{k}}\Big)=0 &\qquad \text{ for } \eta>0.
}
\end{prop}

\begin{proof}
The first equation in \eq{E:CEf} is equivalent to
\Eq{*}{
  \int_0^c f\Big(\frac1{x}\Big)\:dx=0,
}
which, by \cite[Theorem 3.4]{PalPas18a} has a unique solution $c$ in the interval $(1,\infty)$. On the other hand, putting $q:=1-\eta$, the second equation in \eq{E:CEf} is equivalent to the $F(1/c,1-\eta)=0$, which, according to Lemma~\ref{lem:F}, also has a unique solution in the interval $(1,\infty)$.

Fix any $K\in(0,c)$. Then there exists $n_K \in \N$ such that
\Eq{*}{
 K<\Ef\limits_{k=1}^n\Big(\frac{\Lambda_n}{\Lambda_k},\lambda_k \Big) \qquad \text{ for all }n>n_K\:.
 }
Equivalently,
\Eq{*}{
0<\sum_{k=1}^n \frac{\lambda_k}{\Lambda_n} f\Big( \frac{\Lambda_n}{K\Lambda_k}\Big) \qquad \text{ for all }n>n_K.
}
Then, with $\varphi_K(x):=f(\tfrac1{Kx})$, we have
\Eq{*}{
0<\sum_{k=1}^n \frac{\lambda_k}{\Lambda_n} \varphi_K\Big( \frac{\Lambda_k}{\Lambda_n}\Big) \qquad \text{ for all }K \in (0,c) \text{ and }n\ge n_K.
}
Observe that $\varphi_K$ is a nonincreasing, continuous and integrable function on $(0,1]$, therefore upon taking the limit $n\to\infty$ and using Proposition~\ref{prop:genA}, it follows that 
\Eq{KK}{
0\leq
\begin{cases}
\int\limits_0^1 \varphi_K(x) dx & \mbox{if } \eta=0, \\[3mm] 
\sum\limits_{k=0}^{\infty} \eta (1-\eta)^k \varphi_K\big((1-\eta)^k\big) & \mbox{if }\eta \in (0,1).
\end{cases}
}

Similarly, for all $L \in (c,+\infty)$, there exist a sequence of integers $n_i\to\infty$, such that, for all $i\in\N$,
\Eq{*}{
0>\sum_{k=1}^{n_i} \frac{\lambda_k}{\Lambda_{n_i}} \varphi_L\Big( \frac{\Lambda_k}{\Lambda_{n_i}}\Big).
}
Upon taking the limit $i\to\infty$ and again using Proposition~\ref{prop:genA}, we obtain that
\Eq{LL}{
0\geq
\begin{cases}
\int\limits_0^1 \varphi_L(x) dx & \mbox{if } \eta=0, \\[3mm] 
\sum\limits_{k=0}^{\infty} \eta (1-\eta)^k \varphi_L\big((1-\eta)^k\big) & \mbox{if }\eta \in (0,1).
\end{cases}
}
Combining the first inequalities from \eq{KK} and \eq{LL}, in the case $\eta=0$, we get
\Eq{*}{
  \int\limits_0^L f\Big(\frac1{x}\Big) dx 
  =\int\limits_0^1 f\Big(\frac1{Lx}\Big) dx 
  \leq0\leq \int\limits_0^1 f\Big(\frac1{Kx}\Big) dx
  =\int\limits_0^K f\Big(\frac1{x}\Big) dx ,
}
while, for $\eta\in(0,1)$, we obtain
\Eq{*}{
\eta F(L^{-1},1-\eta)
&=\sum\limits_{k=0}^{\infty} \eta (1-\eta)^k f\Big(\frac1{L(1-\eta)^k}\Big)\\
  &\leq0\leq\sum\limits_{k=0}^{\infty} \eta (1-\eta)^k f\Big(\frac1{K(1-\eta)^k}\Big)=\eta F(K^{-1},1-\eta).
}
If we now take the common limits $K \nearrow c$ and $L \searrow c$, and we use the continuity of $F$ established in Lemma~\ref{lem:F}, we get \eq{E:CEf}.
\end{proof}

\section{Applications}
Now we are going to present some weighted Hardy constants for quasiarithmetic means. It is well known that for $\pi_p(x):=x^p$ if $p\ne 0$ and $\pi_0(x):=\ln x$ equality $\QA{\pi_p}=\P_p$ holds. Furthermore, the comparability problem within this family can be (under natural smoothness assumptions) boiled down to pointwise comparability of the mapping $f\mapsto \frac{f''}{f'}$ (cf.\ \cite{HarLitPol34}). More precisely, we have

\begin{prop}\label{Mik}
Let $I \subset \R$ be an interval, $f,\,g \colon I \to \R$ be twice differentiable functions having nowhere vanishing  
first derivatives. Then the following two conditions are equivalent
 \begin{itemize}
 \item[\upshape{(i)}] $\QA{f}(x_1,\dots,x_n)\leq\QA{g}(x_1,\dots,x_n)$ 
 for all $n\in\N$ and vector $(x_1,\dots,x_n)\in I^n$;
 \item[\upshape{(ii)}] $\frac{f''(x)}{f'(x)}\leq \frac{g''(x)}{g'(x)}$ for all $x\in I$.
 \end{itemize}
\end{prop}

In a special case $I\subseteq\R_+$ condition (ii) can be equivalently written as
\Eq{*}{
  \chi_f(x):=\frac{xf''(x)}{f'(x)}+1\leq\frac{xg''(x)}{g'(x)}+1=:\chi_g(x) \qquad(x\in I).
}
It is easy to verify that the equality $\chi_{\pi_p}\equiv p$ holds for 
all $p\in\R$. Therefore, in view of Proposition~\ref{Mik}, we have
\Eq{*}{
\P_q=\QA{\pi_q} \le \QA{f} \le \QA{\pi_p}=\P_p,
}
where $q:=\inf_I\chi_f$ and $p:=\sup_I\chi_f$, moreover 
these parameters are sharp. In other words, the operator $\chi_{(\cdot)}$ could be applied to embed quasiarithmetic means into the scale of power means (cf. \cite{Pas13}). 

This fact will be used to establish some weighted Hardy constants for quasiarithmetic means. Our main idea is to compare a quasiarithmetic mean with a suitable power mean. As a matter of fact, this is not so restrictive as it seams to be at first glance. Namely, Mulholland \cite{Mul32} proved that a quasiarithmetic mean is Hardy if and only if it is majorized up to a constant number by some power mean with parameter strictly smaller than one. Throughout this section, we will use the already introduced notation $\Lambda_n:=\lambda_1+\cdots+\lambda_n$.

\begin{prop}\label{prop:QA}
Let $(\lambda_n)\in W_0$ such that $\Lambda_n\to\infty$ and $\lim_{n \to \infty} \lambda_n/\Lambda_n=:\eta$ exists. Let $f \colon \R_+ \to \R$ be a twice continuously differentiable function with a nonvanishing first derivative and define 
\Eq{*}{
q:= \liminf_{x \to 0^+} \chi_f(x) \le \limsup_{x \to 0^+} \chi_f(x) =:p.
}
Assume that $p<1$. Then, for all $x \in \R_+$,
\Eq{CC}{
C(q,\eta)
\le \liminf_{n \to \infty} \frac{\Lambda_n}{x} \QA{f} &\bigg(\Big(\frac{x}{\Lambda_1},\frac{x}{\Lambda_2},\dots,\frac{x}{\Lambda_n}\Big),(\lambda_1,\dots,\lambda_n)\bigg)\\
\le \limsup_{n \to \infty} \frac{\Lambda_n}{x} \QA{f} &\bigg(\Big(\frac{x}{\Lambda_1},\frac{x}{\Lambda_2},\dots,\frac{x}{\Lambda_n}\Big),(\lambda_1,\dots,\lambda_n)\bigg)
\le C(p,\eta),
}
where the function $C:(-\infty,1)\times[0,1)\to\R$ is defined by
\Eq{defC}{
C(r,\eta):=\begin{cases}
\left( \dfrac{\eta}{1-(1-\eta)^{1-r}} \right) ^{1/r}, &\quad \eta \in (0,1) \text{ and } r \ne 0;\\[4mm]
(1-\eta)^{1-1/\eta}, &\quad \eta \in (0,1) \text{ and } r = 0;\\[1mm]
(1-r) ^{-1/r}, &\quad \eta = 0 \text{ and } r \ne 0; \\
e, &\quad \eta = 0 \text{ and } r = 0.
\end{cases}
}
\end{prop}
\begin{proof}
It is elementary to see that $C$ is a continuous function which is strictly increasing in its first variable.

Following the lines of proof of \cite[Theorem 3.1]{PalPas19a} we get that for all $r \in (p,1) \setminus \{0\}$,
\Eq{*}{
U:= \limsup_{n \to \infty} \frac{\Lambda_n}{x} \cdot \QA{f} \bigg(\Big(\frac{x}{\Lambda_1}&,\frac{x}{\Lambda_2},\dots,\frac{x}{\Lambda_n}\Big),(\lambda_1,\dots,\lambda_n)\bigg) \\
&\le \limsup_{n \to \infty} \frac{\Lambda_n}{x} \cdot \P_r \bigg(\Big(\frac{x}{\Lambda_1},\frac{x}{\Lambda_2},\dots,\frac{x}{\Lambda_n}\Big),(\lambda_1,\dots,\lambda_n)\bigg).
}
Therefore, as $\P_r$ is homogeneous, we obtain
\Eq{*}{
U \le \limsup_{n \to \infty}  \P_r \bigg(\Big(\frac{\Lambda_n}{\Lambda_1},\frac{\Lambda_n}{\Lambda_2},\dots,\frac{\Lambda_n}{\Lambda_n}\Big),(\lambda_1,\dots,\lambda_n)\bigg)
= \limsup_{n \to \infty}  \left(\sum_{k=1}^n \frac{\lambda_k}{\Lambda_n} \cdot \left(\frac{\Lambda_k}{\Lambda_n} \right)^{-r} \right)^{1/r}.
}
Thus, due to Corollary~\ref{cor:A}, we obtain $U \le C(r,\eta)$ for $r \in (p,1)\setminus\{0\}$. Now, as the function $C$ is continuous, we can pass the limit $r \searrow p$ and obtain $U \le C(p,\eta)$. The verification of the left hand side inequality in \eq{CC} is completely analogous.
\end{proof}

We will now establish some $\lambda$-Hardy constants in a family of quasiarithmetic means.

\begin{cor}\label{cor:QA}
Let $(\lambda_n)\in W_0$ such that $\Lambda_n\to\infty$ and $\big(\tfrac{\lambda_n}{\Lambda_n}\big)_{n=1}^{\infty}$ is nonincreasing with a limit $\eta\in[0,1)$. Let $f \colon \R_+ \to \R$ be a twice continuously differentiable function with a nonvanishing first derivative, such that the limit 
\Eq{*}{
p:= \lim_{x \to 0^+} \chi_f(x)
}
exists, is smaller than 1, and $\chi_f(x) \le p$ for all $x \in \R_+$. Then $\Hc[\lambda]{\QA{f}}=C(p,\eta)$, where the function $C$ was defined by \eq{defC}.
\end{cor}

\begin{proof}
By Corollary~\ref{cor:HardyconstantforKedlaya} and Proposition~\ref{prop:QA} we have
 \Eq{*}{
 \Hc[\lambda]{\QA{f}} \ge \sup_{x>0} \liminf_{n \to \infty} \frac{\Lambda_n}{x} \QA{f} \bigg(\Big(\frac{x}{\Lambda_1},\dots,\frac{x}{\Lambda_n}\Big),(\lambda_1,\dots,\lambda_n)\bigg)=C(p,\eta).
 }
Furthermore, by $\chi_f(x)\le p$ we get $\QA{f} \le \P_p$ so 
\Eq{*}{
\Hc[\lambda]{\QA{f}}\le \Hc[\lambda]{\P_p}.
}
But $\P_p$ is repetition invariant and concave, thus it is a $\lambda$-Kedlaya mean (in the sense of our paper \cite{PalPas18b}). Thus, by Proposition~\ref{prop:BJ} and Proposition~\ref{prop:QA}, 
\Eq{*}{
\Hc[\lambda]{\P_p}\le \liminf_{n \to \infty} \frac{\Lambda_n}{x} \P_p \bigg(\Big(\frac{x}{\Lambda_1},\dots,\frac{x}{\Lambda_n}\Big),(\lambda_1,\dots,\lambda_n)\bigg) = C(p,\eta).
}
Binding all these inequalities, we get
\Eq{*}{
C(p,\eta)\le \Hc[\lambda]{\QA{f}} \le \Hc[\lambda]{\P_p} \le C(p,\eta),
}
which implies $\Hc[\lambda]{\QA{f}}=C(p,\eta)$.
\end{proof}

\begin{thm}\label{thm:HDM}
Let $(\lambda_n)\in W_0$ such that $\Lambda_n\to\infty$ and $\big(\tfrac{\lambda_n}{\Lambda_n}\big)_{n=1}^{\infty}$ is nonincreasing with limit $\eta\in[0,1)$. Let $f:\R_+\to\R$ be a concave function such that $\sign(f(x))=\sign(x-1)$ holds for all $x\in\R_+$. Then the homogeneous quasideviation mean $\E_f$ is $\lambda$-Hardy if and only if function $x\mapsto f(1/x)$ is integrable over $(0,1]$. In the latter case, $c:=\Hc[\lambda]{\E_f}$ is the unique solution of the equation \eq{E:CEf}.
\end{thm}

\begin{proof}
Assume that $\E_f$ is $\lambda$-Hardy. Then, by Corollary~\ref{cor:HardyconstantforKedlaya},
\Eq{*}{
  \liminf_{n\to\infty}\Ef_{k=1}^n\bigg(\frac{\Lambda_n}{\Lambda_k},\lambda_k\bigg)=\Est[\lambda]{\E_f}\leq\Hc[\lambda]{\E_f}<\infty.
}
Then, there exists a strictly increasing sequence of integers $(n_i)$ such that
\Eq{*}{
  \Ef_{k=1}^{n_i}\bigg(\frac{\Lambda_{n_i}}{\Lambda_k},\lambda_k\bigg)<\Hc[\lambda]{\E_f}+1=:K \qquad(i\in\N),
}
which is equivalent to the inequality
\Eq{*}{
   \sum_{k=1}^{n_i}\frac{\lambda_k}{\Lambda_{n_i}} f\bigg(\frac{\Lambda_{n_i}}{K\Lambda_k}\bigg)<0
   \qquad(i\in\N).
}
Applying now Proposition~\ref{prop:genA} for the nonincreasing function $\varphi(x):=f\big(\tfrac1{Kx}\big)$, upon taking the limit $i\to\infty$, in the case when $\eta=0$, it follows that
\Eq{*}{
  \int_0^1f\Big(\frac{1}{Kx}\Big)dx\leq0,
}
while in the case $\eta\in(0,1)$, we get that
\Eq{*}{
  \sum_{k=0}^{\infty}\eta(1-\eta)^kf\Big(\frac{1}{K(1-\eta)^k}\Big)\leq0.
}
The first inequality implies that $\varphi$ is integrable over $(0,1]$, hence the mapping $x\mapsto f(1/x)$ is also integrable on $(0,1]$. In view of Lemma~\ref{lem:eq-conv}, the same conclusion is derived from the second inequality.

In the rest of the proof, assume that the mapping $x\mapsto f(1/x)$ is also integrable on $(0,1]$. Obviously $\E_f$ is a homogeneous, symmetric and continuously weighted mean. Moreover, in view of Lemma~\ref{lem:MIA2.3}, $\E_f$ is monotone and Jensen concave. Thus, by Proposition~\ref{prop:BJ}, $\Hc[\lambda]{\E_{f}}=\Est[\lambda]{\E_{f}}$. Consequently, applying Proposition~\ref{prop:CEf}, one obtains that $c=\Hc[\lambda]{\E_{f}}$ is a unique and finite solution of equation \eq{E:CEf}, indeed. In particular, this yields, that $\E_f$ is a $\lambda$-Hardy mean.
\end{proof}

An interesting consequence of the previous result is that a homogeneous quasideviation mean $\E_f$ is $\lambda$-Hardy (where $\lambda$ is like above) if and only if it is $\vone$-Hardy.

One of our main results is stated in the subsequent theorem.

\begin{thm}
Let $(\lambda_n)\in W_0$ such that $\Lambda_n\to\infty$ and $\big(\tfrac{\lambda_n}{\Lambda_n}\big)_{n=1}^{\infty}$ is nonincreasing with limit $\eta\in[0,1)$. Let $E : I \times I \to \R$ be a normalizable quasideviation such that $E^*$ is concave. Assume that, for all $x \in \R_+$, $\lim_{t \to 0}E^*(xt, t) = 0$ and define $h_E:\R_+\to\R$
by \eq{hE}. Then the quasideviation mean $\D_E$ is $\lambda$-Hardy if and only if the mapping $x\mapsto h_E(1/x)$ is integrable over $(0,1]$ and in this case, $c:=\Hc[\lambda]{\D_E}$ is a unique solution of \eqref{E:CEf} with $f:=h_E$.
\end{thm}

\begin{proof}
First, by Lemma~\ref{lem:AMH6.3} we know that $f:=h_E$ is correctly defined. Furthermore it is nondecreasing on $(0,\infty)$, strictly increasing on $(0,1)$, and admits the sign property $\sign(f(x))=\sign(x-1)$ and $\E_f$ is a homogeneous quasideviation mean. 

First assume that $\D_E$ is a $\lambda$-Hardy mean. Then, by Theorem~\ref{thm:M*}, $\big(\D_E\big)_\#$ is also a $\lambda$-Hardy mean. On the other hand, Lemma~\ref{lem:AMH6.3} implies that $\big(\D_E\big)_\#=\E_{f}$,
hence, we get that $\E_{f}$ is a $\lambda$-Hardy mean, too. By the previous theorem, this implies that the mapping $x\mapsto f(1/x)$ is integrable over $(0,1]$.

In the rest of the proof, assume that the mapping $x\mapsto f(1/x)$ is integrable over $(0,1]$. In view of Proposition~\ref{prop:CEf}, we have that $c:=\Hc[\lambda]{\E_f}$ is a unique solution of \eqref{E:CEf}.
On the other hand, by Theorem~\ref{thm:M*}, we have that $\Hc[\lambda]{\D_E}=\Hc[\lambda]{\E_f}$, which yields that $\Hc[\lambda]{\D_E}=c$.
\end{proof}

\begin{cor}\label{cor:Gini}
Let $(\lambda_n)\in W_0$ such that $\Lambda_n\to\infty$ and $\big(\tfrac{\lambda_n}{\Lambda_n}\big)_{n=1}^{\infty}$ is nonincreasing with a limit $\eta\in[0,1)$. Let $p,\,q \in \R$, $\min(p,q)\le 0 \le \max(p,q)<1$. Then 
\Eq{*}{
\Hc[\lambda]{\G_{p,q}}=
\begin{cases}
\bigg(\dfrac{1-(1-\eta)^{1-q}}{1-(1-\eta)^{1-p}}\bigg)^{\tfrac{1}{p-q}}, &\quad \eta \in (0,1) \text{ and } p\ne q;\\[1mm]
\bigg(\dfrac{1-q}{1-p}\bigg)^{\tfrac{1}{p-q}}, &\quad \eta = 0 \text{ and } p\ne q; \\[4mm]
(1-\eta)^{1-1/\eta}, &\quad \eta \in (0,1) \text{ and } p=q= 0;\\[1mm]
e, &\quad \eta = 0 \text{ and } p=q=0.
\end{cases}
}
\end{cor}
\begin{proof}
 Fix $p,q$ like above. In the  case $p=0$ (resp.\ $q=0$), we have $\G_{p,q}=\P_q$ (resp. $\G_{p,q}=\P_p$) and the assertion is implied by Corollary~\ref{cor:QA}. As $\G_{p,q}=\G_{q,p}$ and the right hand side is symmetric, we can assume that $p<0<q<1$. 
 
 Observe that Gini means are homogeneous deviation means -- more precisely $\G_{p,q}=\E_f$ with $f(x)=\frac{x^p-x^q}{p-q}$. The condition $p<0<q<1$ implies that $f$ is concave, satisfies the sign condition and the mapping $x\mapsto f(1/x)$ is integrable. Therefore, Theorem~\ref{thm:HDM} yields that $\Hc[\lambda]{\G_{p,q}}$ is the unique solution $c$ of equation \eq{E:CEf}.
 
 Let us now split our considerations into two parts. For $\eta=0$, we have
 \Eq{*}{
 0=\int_0^1 f\Big(\frac{1}{cx}\Big)dx=\int_0^1 \frac{c^{-p}}{p-q} x^{-p} -\frac{c^{-q}}{p-q} x^{-q}dx= \frac{1}{p-q} \cdot  \Big(\frac{c^{-p}}{1-p}-\frac{c^{-q}}{1-q}\Big),
 }
which, after an easy transformation, is equivalent to $c=\big(\tfrac{1-q}{1-p}\big)^{1/(p-q)}$.

For $\eta>0$, we need to solve the second equation of \eq{E:CEf}, which in our setting states
\Eq{*}{
\frac{1}{p-q} \cdot \sum\limits_{k=0}^{\infty} (1-\eta)^k \big(c^{-p} (1-\eta)^{-kp} - c^{-q} (1-\eta)^{-kq}\big)=0.
}
As $\eta \in (0,1)$, we can calculate the sums of the geometric series to obtain
\Eq{*}{
\frac{1}{p-q} \cdot \Big(\frac{c^{-p}}{1-(1-\eta)^{1-p}}-\frac{c^{-q}}{1-(1-\eta)^{1-q}}\Big)=0.
}
As $p \ne q$ and $c \ne 0$, it implies 
\Eq{*}{
c^{p-q}=\frac{1-(1-\eta)^{1-q}}{1-(1-\eta)^{1-p}},
}
and yields the assertion in the last case.
\end{proof}

\begin{xrem}
It is worth mentioning that (exept the case $p=q=0$) we have the equality $\Hc[\lambda]{\G_{p,q}}=C(p,\eta)^{p/(p-q)} C(q,\eta)^{q/(q-p)}$. As a matter of fact, this assertion could be obtained using a similar identity: $\G_{p,q}(x,\lambda)=\P_p(x,\lambda)^{p/(p-q)}\P_q(x,\lambda)^{q/(q-p)}$, which is valid for all $p,q \in\R$, $p \ne q$ and all admissible pairs $(x,\lambda)$. 
\end{xrem}


\begin{thebibliography}{10}

\bibitem{Baj58}
M.~Bajraktarević.
\newblock {Sur une équation fonctionnelle aux valeurs moyennes}.
\newblock {\em Glasnik Mat.-Fiz. Astronom. Društvo Mat. Fiz. Hrvatske Ser.
  II}, 13:243–248, 1958.

\bibitem{BerDoe15}
F.~Bernstein and G.~Doetsch.
\newblock {Zur {T}heorie der konvexen {F}unktionen}.
\newblock {\em Math. Ann.}, 76(4):514–526, 1915.

\bibitem{Car32}
T.~Carleman.
\newblock {Sur les fonctions quasi-analitiques}.
\newblock {\em Conférences faites au cinquième congrès des mathématiciens
  scandinaves, Helsinki}, page 181–196, 1932.

\bibitem{Cop27}
E.~T. Copson.
\newblock {Note on Series of Positive Terms}.
\newblock {\em J. London Math. Soc.}, s1-2(1):9–12, 1927.

\bibitem{Dar71b}
Z.~Daróczy.
\newblock {A general inequality for means}.
\newblock {\em Aequationes Math.}, 7(1):16–21, 1971.

\bibitem{Dar72b}
Z.~Daróczy.
\newblock {Über eine {K}lasse von {M}ittelwerten}.
\newblock {\em Publ. Math. Debrecen}, 19:211–217 (1973), 1972.

\bibitem{DarLos70}
Z.~Daróczy and L.~Losonczi.
\newblock {Über den {V}ergleich von {M}ittelwerten}.
\newblock {\em Publ. Math. Debrecen}, 17:289–297 (1971), 1970.

\bibitem{DarPal82}
Z.~Daróczy and Zs. Páles.
\newblock {On comparison of mean values}.
\newblock {\em Publ. Math. Debrecen}, 29(1-2):107–115, 1982.

\bibitem{DarPal83}
Z.~Daróczy and Zs. Páles.
\newblock {Multiplicative mean values and entropies}.
\newblock In {\em {Functions, series, operators, Vol. I, II (Budapest, 1980)}},
  page 343–359. North-Holland, Amsterdam, 1983.

\bibitem{DucMcG03}
J.~Duncan and C.~M. McGregor.
\newblock {Carleman's Inequality}.
\newblock {\em Amer. Math. Monthly}, 110(5):424–431, 2003.

\bibitem{Ell26}
E.~B. Elliott.
\newblock {A simple exposition of some recently proved facts as to
  convergency}.
\newblock {\em J. London Math. Soc.}, 1:93–96, 1926.

\bibitem{Gin38}
C.~Gini.
\newblock {{D}i una formula compressiva delle medie}.
\newblock {\em Metron}, 13:3–22, 1938.

\bibitem{Har20a}
G.~H. Hardy.
\newblock {Note on a theorem of Hilbert.}
\newblock {\em {Math. Z.}}, 6:314–317, 1920.

\bibitem{HarLitPol34}
G.~H. Hardy, J.~E. Littlewood, and G.~Pólya.
\newblock {\em {Inequalities}}.
\newblock Cambridge University Press, Cambridge, 1934.
\newblock (first edition), 1952 (second edition).

\bibitem{Kno28}
K.~Knopp.
\newblock {Über {R}eihen mit positiven {G}liedern}.
\newblock {\em J. London Math. Soc.}, 3:205–211, 1928.

\bibitem{KufMalPer07}
A.~Kufner, L.~Maligranda, and L.E. Persson.
\newblock {\em {The Hardy Inequality: About Its History and Some Related
  Results}}.
\newblock Vydavatelsk\`y servis, 2007.

\bibitem{Lan21}
E.~Landau.
\newblock {A note on a theorem concerning series of positive terms}.
\newblock {\em J. London Math. Soc.}, 1:38–39, 1921.

\bibitem{Los70a}
L.~Losonczi.
\newblock {Über den {V}ergleich von {M}ittelwerten die mit
  {G}ewichtsfunktionen gebildet sind}.
\newblock {\em Publ. Math. Debrecen}, 17:203–208 (1971), 1970.

\bibitem{Los71a}
L.~Losonczi.
\newblock {Subadditive {M}ittelwerte}.
\newblock {\em Arch. Math. (Basel)}, 22:168–174, 1971.

\bibitem{Los71c}
L.~Losonczi.
\newblock {Subhomogene {M}ittelwerte}.
\newblock {\em Acta Math. Acad. Sci. Hungar.}, 22:187–195, 1971.

\bibitem{Los71b}
L.~Losonczi.
\newblock {Über eine neue {K}lasse von {M}ittelwerten}.
\newblock {\em Acta Sci. Math. (Szeged)}, 32:71–81, 1971.

\bibitem{Los73a}
L.~Losonczi.
\newblock {General inequalities for nonsymmetric means}.
\newblock {\em Aequationes Math.}, 9:221–235, 1973.

\bibitem{Los77}
L.~Losonczi.
\newblock {Inequalities for integral mean values}.
\newblock {\em J. Math. Anal. Appl.}, 61(3):586–606, 1977.

\bibitem{Mul32}
P.~Mulholland.
\newblock {On the generalization of {H}ardy's inequality}.
\newblock {\em J. London Math. Soc.}, 7:208–214, 1932.

\bibitem{Pas13}
P.~Pasteczka.
\newblock {When is a family of generalized means a scale?}
\newblock {\em Real Anal. Exchange}, 38(1):193–209, 2012/13.

\bibitem{PecSto01}
J.~E. Pečarić and K.~B. Stolarsky.
\newblock {Carleman's inequality: history and new generalizations}.
\newblock {\em Aequationes Math.}, 61(1–2):49–62, 2001.

\bibitem{Pal82a}
Zs. Páles.
\newblock {Characterization of quasideviation means}.
\newblock {\em Acta Math. Acad. Sci. Hungar.}, 40(3-4):243–260, 1982.

\bibitem{Pal83b}
Zs. Páles.
\newblock {On complementary inequalities}.
\newblock {\em Publ. Math. Debrecen}, 30(1-2):75–88, 1983.

\bibitem{Pal84a}
Zs. Páles.
\newblock {Inequalities for comparison of means}.
\newblock In W.~Walter, editor, {\em {General Inequalities, 4 (Oberwolfach,
  1983)}}, volume~71 of {\em {International Series of Numerical Mathematics}},
  page 59–73. Birkhäuser, Basel, 1984.

\bibitem{Pal85a}
Zs. Páles.
\newblock {Ingham {J}essen's inequality for deviation means}.
\newblock {\em Acta Sci. Math. (Szeged)}, 49(1-4):131–142, 1985.

\bibitem{Pal88a}
Zs. Páles.
\newblock {General inequalities for quasideviation means}.
\newblock {\em Aequationes Math.}, 36(1):32–56, 1988.

\bibitem{Pal88d}
Zs. Páles.
\newblock {On a {P}exider-type functional equation for quasideviation means}.
\newblock {\em Acta Math. Hungar.}, 51(1-2):205–224, 1988.

\bibitem{Pal88e}
Zs. Páles.
\newblock {On homogeneous quasideviation means}.
\newblock {\em Aequationes Math.}, 36(2-3):132–152, 1988.

\bibitem{Pal89b}
Zs. Páles.
\newblock {A {H}ahn-{B}anach theorem for separation of semigroups and its
  applications}.
\newblock {\em Aequationes Math.}, 37(2-3):141–161, 1989.

\bibitem{PalPas19b}
Zs. Páles and P.~Pasteczka.
\newblock On the homogenization of means.
\newblock {\em Acta Math. Hungar.}

\bibitem{PalPas16}
Zs. Páles and P.~Pasteczka.
\newblock {Characterization of the {H}ardy property of means and the best
  {H}ardy constants}.
\newblock {\em Math. Ineq. Appl.}, 19:1141–1158, 2016.

\bibitem{PalPas18b}
Zs. Páles and P.~Pasteczka.
\newblock {On {K}edlaya type inequalities for weighted means}.
\newblock {\em J. Inequal. Appl.}, 2018(99), 2018.

\bibitem{PalPas18a}
Zs. Páles and P.~Pasteczka.
\newblock {On the best {H}ardy constant for quasi-arithmetic means and
  homogeneous deviation means}.
\newblock {\em Math. Inequal. Appl.}, 21:585–599, 2018.

\bibitem{PalPas19a}
Zs. Páles and P.~Pasteczka.
\newblock {On {H}ardy type inequalities for weighted means}.
\newblock {\em Banach J. Math. Anal.}, 13:217–233, 2019.

\end{thebibliography}

\def\cprime{$'$} \def\R{\mathbb R} \def\Z{\mathbb Z} \def\Q{\mathbb Q}
  \def\C{\mathbb C}

\end{document}